\documentclass{article}
\usepackage[english]{babel}
\usepackage[utf8]{inputenc}
\usepackage[backend = biber,url=false,isbn=false,doi=false,maxnames=4, eprint=false]{biblatex}
\usepackage{csquotes}
\usepackage{todonotes}
\usepackage{subcaption}
\usepackage[normalem]{ulem}
\usepackage{amssymb}
\usepackage{amsmath}
\usepackage{ulem}
\usepackage{amsthm} 
\usepackage{exscale}
\usepackage{mathtools}
\usepackage{enumitem}

\renewbibmacro{in:}{%
  \ifentrytype{article}{}{%
  \printtext{\bibstring{in}\intitlepunct}}}
\DeclareFieldFormat[article,unpublished,book]{pages}{#1}  
\DeclareFieldFormat[article,incollection,inbook,inproceedings,misc,unpublished]{title}{#1}

\newcommand{\dd}{\text{d}}
\newcommand{\av}[1]{\left<#1\right>}
\newtheorem{theorem}{Theorem}[section]
\newtheorem{definition}[theorem]{Definition}
\newtheorem{proposition}[theorem]{Proposition}
\newtheorem{remark}[theorem]{Remark}

\renewcommand{\P}{\mathbb{P}}
\newcommand{\R}{\mathbb{R}}
\newcommand{\E}{\mathbb{E}}
\newcommand{\N}{\mathbb{N}}

\newcommand{\F}{F}
\newcommand{\C}{\mathcal C}
\newcommand{\cM}{\mathcal M}
\newcommand{\cN}{\mathcal N}

\newtheorem{lemma}[theorem]{Lemma}
\newtheorem{corollary}[theorem]{Corollary}
\newtheorem{assumption}[theorem]{Assumption}

\begin{document}

\title{Well-Posedness for Dean-Kawasaki Models of Vlasov-Fokker-Planck Type}
\author{Fenna M\"uller, Max von Renesse, Johannes Zimmer}

\maketitle

\abstract{We consider systems of interacting particles which are described by a second order Langevin
  equation. The class of equations considered includes the situation where the particle evolution is governed
  by Hamiltonian dynamics with additional damping and noise satisfying a fluctuation-dissipation
  relation. Also covered are systems of two equations describing an evolution of interacting agents, as
  arising in several descriptions of active matter, including models for flocking and swarming. We first show
  that such particle systems can be represented \emph{exactly} by so-called equations of fluctuating
  hydrodynamics, which in this case are stochastic versions of a Vlasov-Fokker-Planck type equation. While the
  derivation given here is simple, it is a blueprint for the rigorous derivation of equations of fluctuating
  hydrodynamics.  We then show a dichotomy previously known for purely diffusive (first order) systems carries
  over to the second order setting considered here: Solutions exist for suitable atomic initial data, in which
  case the solution is, properly scaled, the empirical density describing the particle system. For smooth
  initial data, however, we prove that no solution exists.
}

\section{Introduction}

Equations of fluctuating hydrodynamics are stochastic partial differential equations (SPDEs) describing the
collective evolution of particles. They are widely used in physics, among others to describe active matter
(e.g.,~\cite{Fodor2022a}). The mathematical analysis of such equations is in its infancy, despite significant
recent progress. Equations of fluctuating hydrodynamics can be thought of as the the mean-field limit of many
particle systems (such as the Vlasov-Fokker-Planck equation) ``plus noise'', where the noise is of a
particular structure, reflecting a fluctuation-dissipation relation. Specifically, fluctuating hydrodynamics
is a mesoscopic description of interacting particle systems (as opposed to the mean-field limit describing
infinitely many particles). We show below that this description of many-particle dynamics is exact and can be
derived mathematically using a martingale formulation. The argument given below is a precise mathematical
derivation, complementing the noise replacement methods used in the physics literature~\cite{Dean1996a}. This
method is applied here to derive new equations describing a large variety of models.

We consider fluctuating hydrodynamics for Vlasov-Fokker-Planck type equations, including hypoelliptic
equations, which describe the evolution of second order Langevin equations (that is, evolution involving
inertia). We show that for purely atomic initial data, such as empirical measures, solution exist in a
martingale sense, while provably no solutions exist for smooth initial data. This dichotomy between existence
and nonexistence was previously only known for purely diffusive
equations~\cite{Konarovskyi2019a,Konarovskyi2020a}, i.e., Fokker-Planck equations describing the evolution of
particles undergoing first order Langevin dynamics. One interpretation is that the equations of fluctuating
hydrodynamics considered here describe exactly the dynamics of an arbitrary number of particles experiencing
second order Langevin dynamics.

For instance, the result of this papers covers the collective evolution of $n$ particles given by second order
Langevin equations of the type
\begin{align*}              
\dd x_t^i &=  v_t^i \dd t,\\
  \dd v_t^i &=  - \gamma v_t^i \dd t %
              -\frac 1 n \sum_{j=1}^n \nabla G(x_t^i - x_t^j) \dd t
              +  \sqrt{\gamma} \dd W_t^i,
\end{align*} 
where $G$ is an interaction potential and $W_t^i$ are independent Brownian motions. More general settings,
including interaction terms depending on the velocities and nonlinear dependence in the first equation, are
considered in Section~\ref{sec.active-matter}. The class of models covered in this setting includes various
models of active matter, see Section~\ref{sec:examples}.

As described in Section~\ref{sec:deriv-equat-fluct}, the system of particles evolving with this system of
equations can, for any $n\in\mathbb{N}$, be equivalently be described by a stochastic Vlasov-Fokker-Planck
equation for the empirical measure density $\mu_t(x,v) = \frac 1 n \sum_{i=1}^n \delta_{x_t^i, v_t^i}$, where
$\delta$ denotes the Dirac delta. %
For the example above, this stochastic partial differential equation (SPDE) reads
\begin{align} 
\label{intro.VLPDK}
  \partial_t \mu_t = %
  -  \nabla_x \cdot (v \mu_t )  + \gamma\nabla_v \cdot (v \mu_t  )
  + \nabla_v \cdot(\mu_t \nabla G \ast \nu_t ) %
  + \frac{\gamma}{2} \Delta_v \mu_t 
  + \nabla_v \cdot (\sqrt{\gamma\mu_t} \dot W_{x,v,t}^v),
\end{align}
where $\nu_t$ is the first marginal of $\mu_t$, i.e., $\nu_t(X) = \mu_t(X,\R^d)$, and
$\dot W_{x,v,t}^v \in \R^d$ is white noise in space $x$, velocity $v$ and time $t$ (denoted in subscripts) in
direction of the $v$ (denoted as superscript). This is an example of a so-called equation of \emph{fluctuating
  hydrodynamics}.

The aforementioned dichotomy, established in Theorem~\ref{thm:mainthm}, shows that equations of fluctuating
hydrodynamics are not well posed in the sense of Cauchy, in fact are very unstable in the sense that even
small smoothing of atomic data with a kernel function immediately renders the equation insoluble. One
interpretation of this result is that equations of fluctuating hydrodynamics are an \emph{exact} encoding of
the underlying particle system, and not meaningful for other initial data.

While particle systems and associated equations of fluctuating hydrodynamics thus encode the same setting and
are in this sense equivalent, they can offer different advantages. For example, the numerical simulation of
equations of fluctuating hydrodynamics can be much more efficient than the simulation of the underlying
particle model~\cite{Helfmann2021a,Cornalba2023a,Cornalba2023bTR}.  Also, equations of fluctuating
hydrodynamics are related to large deviation principles in a way which reveals the geometric structure of the
associated hydrodynamic limit~\cite[Section 3.5]{Jack2014a}. In this way, fluctuating hydrodynamics can be
useful to identify a path integral representation of the dynamics. Equations of fluctuating hydrodynamics are
also central to learn the ``thermodynamic'' evolution operator (such as the Wasserstein operator for
diffusion) from particle data~\cite{Li2019a}.

\subsection{Context and related work}

We first review the situation for purely diffusive (Fokker-Planck) type of equations, as it is much better
studied. The most prominent purely diffusive example is the Dean-Kawasaki equation, which is a field-theoretic
mesoscopic description of interacting particle systems obeying overdamped (first order) Langevin
dynamics~\cite{Dean1996a,Kawasaki1973a}, cf.~\cite{illien2025} for a recent survey. The resulting stochastic
partial differential equation (SPDE) is a singular supercritical SPDE with non-Lipschitz multiplicative
noise. It is known~\cite{Konarovskyi2019a,Konarovskyi2020a} that the equation is well-posed only for initial
data given by suitable atomic measures, such as empirical measures.

After the results of~\cite{Konarovskyi2019a,Konarovskyi2020a}, it was initially not clear whether the
dichotomy (existence for atomic initial data, nonexistence for smooth initial data) is a peculiarity of purely
diffusive Dean-Kawasaki equations. Here we show that this phenomenon holds for a very wide class of equations
of fluctuating hydrodynamics (in fact, we are not aware of an equation of fluctuating hydrodynamics in the
sense considered here where it does not hold).

As mentioned above, the dichotomy has immediate practical implications. It shows that the equations of
fluctuating hydrodynamics, including~\eqref{intro.VLPDK}, covered by this result are very unstable: Every
attempt to smoothen the initial data with a (possibly small) mollifier immediately destroys the existence of a
solution.  One possibility is thus to consider smoothened versions of the Dean-Kawasaki equation, as analyzed
in~\cite{Cornalba2019a, Djurdjevac2024a, Fehrman2024a, Mariani2010a}. Then it can be shown that more regular,
spatially diffuse solutions exist. For equations of fluctuating hydrodynamics without regularisation in the
noise, however, particle methods need to be employed in the simulation, as in~\cite{Embacher2018a,Huang2025a};
methods assuming the existence of a smooth solution profile such as finite difference or finite volume methods
are not applicable. A rigorous justification of diffusive fluctuating hydrodynamics for $N$ particles by
numerical discretizations to arbitrary order in $N^{-1}$ is given in~\cite{Cornalba2023a}. Furthermore,
see~\cite{Angeli2024aTR} for the approximation of McKean-Vlasov equations with additive noise by weighted
particle systems.

We close this section with a brief, and certainly incomplete, discussion of equations of fluctuating
hydrodynamics in different application areas. Variants of the Dean-Kawasaki equation~\cite{Dean1996a} as the
prototypical equation of fluctuating hydrodynamics are widely used in very different application areas. In
fluid mechanics, they appear as models of hydrodynamic fluctuation theory, for example for thin film
models~\cite{Grun2006a} and colloidal nucleation theory~\cite{Lutsko2012a}. They are also used to study
pattern formation in bacterial colonies~\cite{Cates2010a} and more generally in active matter
theory~\cite{Cates2015a}, for example to investigate motility-induced phase transitions. Dean-Kawasaki type
models exist in dynamical density functional theory~\cite{Vrugt2020b}. There are also Dean-Kawasaki models for
systems, such as the Keller-Segel model~\cite{Martini2024aTR}. For Dean-Kawasaki type equations in dynamic
density functional theory see~\cite{Donev2014b}. Dean-Kawasaki equations also appear as models for crowd
behaviour, as alternative to agent-based simulations~\cite{Helfmann2021a}. The stochastic gradient descent
used in Machine Learning can be viewed as a Dean--Kawasaki type process, as shown and analysed
in~\cite{Rotskoff2022a,Gess2022aTR}.

\nocite{Bressloff2024a}

Numerical evidence for the effectiveness of the Dean-Kawasaki field-theoretic descriptions of many particle
systems is presented in~\cite{Cornalba2023bTR} and~\cite{Wehlitz2024aTR}, a rigorous error analysis is carried
out in~\cite{Cornalba2023a,Cornalba2023b,Cornalba2023bTR}. Rigidity breaks down in case of singular
interactions~\cite{Andres2010a,Renesse2009a}. In~\cite{Dello-Schiavo2024aTR}, Dello Schiavo considers
Dean—Kawasaki-type equations induced from symmetric reversible diffusions with singular drift and a large
class of interactions, extending~\cite{Dello-Schiavo2022a} by means of Dirichlet forms.

\section{Derivation of equations of fluctuating hydrodynamics}
\label{sec:deriv-equat-fluct}

We now sketch the derivation of an equation of type~\eqref{intro.VLPDK} from a particle system undergoing
Langevin dynamics. The approach we present here is straightforward, but has in our view two benefits: Firstly,
it puts the sometimes indirect derivation of related equations in the purely diffusive setting in the physics
literature on firm mathematical ground.  Secondly, it shows that the resulting system of fluctuating
hydrodynamics is an \emph{exact} representation of the underlying particle system, without any approximations
involved.

To work with a specific situation, we consider the case of particles obeying
\begin{align}
  \dd x_t^i &= v_t^i\dd t, \label{deriv-a}\\
  \dd v_t^i &= -\frac 1 n \sum_{j=1}^n \nabla H(v_t^i - v_t^j) \dd t + \dd W_t^i, \label{deriv-b}
\end{align}
for $i = 1, \ldots, n$, where $x_t^i \in \R^d$ and $v_t^i \in \R^d$. Here $W_t^i$ are independent Brownian
motions.  We describe the evolution with a probability measure
\begin{equation}
  \label{deriv-delta}
  \mu_t(x,v) = \frac 1 n \sum_{i=1}^n \delta_{(x_t^i, v_t^i)}.
\end{equation}
This seems very natural, but marks the crucial departure from previous
work~\cite{Cornalba2021a,Cornalba2019a,Cornalba2020a} on such systems, where reduced models, where position
and time are the only independent variables. We chose a description of the system by a density which depends
on the position and the velocity of the particles, $(x,v)\in \R^d\times \R^d$.

In the physics literature, the derivation given by Dean~\cite{Dean1996a} for purely diffusive particles now
applies the It\^o formula for the measure $\tilde \mu_t(x) = \sum_{i=1}^n \delta_{x_t^i}$. This results in an
expression which is not closed, as the noise depends on the position of the individual particles
$\delta_{x_t^i}$, rather than $\tilde \mu_t$. In a clever but non-rigorous step, this noise is replaced with a
statistically identical one, which has the same correlations as the noise obtained in the It\^o formula. This
resulting noise is $\nabla_x \cdot (\sqrt{\mu_t} \dot W_{x,t})$, where $W_{x,t}$ is vector-valued space-time
white noise. This noise term is the ``Dean-Kawasaki noise'' of pure diffusion, i.e., the analogue
of~\eqref{intro.VLPDK} for pure diffusion, with $\gamma=1$.

We provide here a rigorous derivation, for the second order Langevin
equations~\eqref{deriv-a}--\eqref{deriv-b}. To this behalf, we test with smooth, compactly supported functions
$\varphi \in C^\infty_c(\R^d \times \R^d)$. We write $\langle \mu_t, \cdot \rangle$ for the space-velocity
integral against $\mu_t$. In the following formal calculation, the second equality uses~\eqref{deriv-a},
\eqref{deriv-b} and the stochastic chain rule from It\^o calculus. In the third equality, $\nu_t$ is the
second marginal of $\mu_t$, i.e., $\nu_t(V) = \mu_t(\R^d,V)$:
\begin{align*}
  \dd \langle \mu_t, \varphi \rangle &= \frac 1 n \sum_{i=1}^n \dd 
                                \varphi(x_t^i, v_t^i) \\
                              &= \frac 1 n \sum_{i=1}^n \nabla_x\varphi(x_t^i, v_t^i) \cdot v_i \dd t 
                                - \frac 1 {n^2} \sum_{i=1}^n \nabla_v\varphi(x_t^i, v_t^i)\cdot\sum_{j=1}^n \nabla H(v_t^i - v_t^j) \dd t \\
                              & \qquad{}
                                + \frac 1 n \sum_{i=1}^n \nabla_v\varphi(x_t^i, v_t^i) \dd W_t^i
                                + \frac 1 {2n} \sum_{i=1}^n \Delta_{v}\varphi(x_t^i, v_t^i) \dd t \\ 
                              &= \int_{\R^d \times \R^d} \nabla_x\varphi(x, v) \cdot v  \mu_t(\dd x \dd v) \dd t -  \int_{\R^d \times \R^d} \nabla_v\varphi(x, v) \cdot (\nabla H \ast \nu_t)(v) \mu_t(\dd x \dd v) \dd t \\
                              &\qquad{}+\frac 1 2 \int_{\R^d \times \R^d}
                                \Delta_{v}\varphi(x, v) \mu_t(\dd x \dd v) \dd t  + M_t(\varphi),
\end{align*}
where $M_\cdot(\varphi)$ is a local martingale corresponding to the sum of It\^o integrals above, with
quadratic variation
\begin{equation*}
  \dd \left[M_\cdot(\varphi)\right]_t =  \frac 1 n \int_{\R^d \times \R^d} \left|\nabla_{v}\varphi(x, v)\right|^2 \mu_t(\dd x \dd v) \dd t.
\end{equation*}
Integration in
time of the previous equality yields that for $t > 0$
\begin{align*}
  \left\langle \mu_t,\varphi\right\rangle & - \left\langle \mu_0, \varphi \right\rangle 
                                      - \int_0^t \left\langle \mu_s, \nabla_x\varphi(x, v)  \cdot v \right\rangle \dd s  \\& +  \int_0^t 
  \left\langle  {\mu_s} , \nabla_v\varphi(x, v) \cdot ( \nabla H \ast \nu_t)(v)\right\rangle \dd s 
  - \frac 1 2 \int_0^t \langle  {\mu_s}, \Delta_{v}\varphi \rangle \dd s 
\end{align*}
is a local martingale with quadratic variation
\begin{equation}    
  \label{deriv-qv}    
  \frac 1 n \int_0^t  \left\langle
    \mu_s, \left|\nabla_{v}\varphi\right|^ 2 \right\rangle 
  \dd s.
\end{equation}
From this martingale problem, one can read off with integration by parts the following stochastic
Vlasov-Fokker-Planck (Dean-Kawasaki) equation:
\begin{equation}
\label{eq:protoeq}    
  \dd \mu_t = \nabla \cdot \left( \begin{pmatrix} -v \\  \nabla H \ast \nu \end{pmatrix} \mu_t\right) \dd t
  + \frac 1 2  \Delta_{v} \mu_t \dd t+ \frac 1 {\sqrt{n}}\nabla_v \cdot\left(\sqrt{\mu_t}\dd W_{x,v,t}^v\right),
\end{equation}
where $W_{x, v,t}^v$ is again noise which is white in space $x$, velocity $v$ and time $t$ in direction of
$v$. The noise term in this SPDE induces in the weak formulation the quadratic variation~\eqref{deriv-qv}.
Rewritten in terms of divergence in space and velocity, the equation reads for $\mu_t = \mu_t(x,v)$ To see
that the noise term in this SPDE induces in fact in the weak formulation the quadratic
variation~\eqref{deriv-qv}, we test the solution of~\eqref{eq:protoeq} with
$\varphi \in C^\infty_c(\R^d \times \R^d)$ and integrate by parts with respect to $x$ and $v$ and find that
the real-valued semi-martingale $\xi_t= \langle \mu_t, \varphi\rangle$ has the quadratic variation
\begin{multline*}
\left[\xi,\xi\right]_t  = \\\frac 1 n \int _0^t \int\limits_{\R^d \times \R^d}\int\limits_{\R^d \times \R^d}  \nabla_{v} \varphi(x,v)\sqrt{\mu_s(x,v)}\cdot  \nabla_{v'} \varphi(x',v')\sqrt{\mu_s(x',v')} \delta_{(x,v)}(\dd(x',v'))\dd (x,y)\dd s 
. \end{multline*} This leads after integration to~\eqref{deriv-qv}; here we use that the components of the
noise vector-field in~\eqref{eq:protoeq} are independent and decorrelated in space and velocity, i.e.,
\[ \mathrm{d}\left[W^k(x,v),W^j(x',v')\right]_t= \delta_{k,j}\,  \delta_{(x,v)}(\dd(x',v'))\dd t. \]

In this article, we consider a scaled version of~\eqref{eq:protoeq} and generalizations thereof. The purpose
of the scaling is to change the prefactor of the noise from $\frac 1 {\sqrt{n}}$ to $1$.  This is achieved by
considering $\tilde \mu_t = \mu_{\alpha t}$ with $\alpha=n$. This results in the equation
\begin{equation}  
\label{eq:protoeqalpha}
  \dd \tilde\mu_t = \alpha \nabla \cdot \left( \begin{pmatrix} -v \\  \nabla H \ast \tilde\nu \end{pmatrix} \tilde\mu_t\right) \dd t
  + \frac \alpha 2  \Delta_{v} \tilde\mu_t \dd t+ \nabla_v \cdot\left(\sqrt{\tilde\mu_t}\dd W_{x,v,t}^v\right),
\end{equation}
where $\tilde \nu$ is the space marginal of $\tilde\mu$.  The main result of this work asserts that in the
space of probability measures, solutions to the SPDE~\eqref{eq:protoeqalpha} exist if and only if $\alpha =n$
for some $n \in \N$, and the initial condition is of the form
$\mu_0 = \frac 1 n \sum_{i=1}^n \delta_{(x_0^i, v_0^i)}$. In this case, the solution is given in the form
\begin{equation*}
    \tilde \mu_t(x,v) = \frac 1 n \sum_{i=1}^n \delta_{(x_{n t}^i, v_{nt}^i)},
\end{equation*}
where $(x_t^i,v_t^i)$ solve the underlying Langevin equations~\eqref{deriv-a}--\eqref{deriv-b}.  (Thus, we
consider empirical measures, hence the scaling is different from the one in~\cite{Dean1996a}.) This result
will be shown in much greater generality holding for an entire class of SPDEs defined in the next section,
which comprises the model considered in this section.

\subsection{Additional comments and relation to other models}

Some remarks about this derivation are in order.  Firstly, the derivation above can be readily extended to
other models which are structurally similar. For instance, one could add additional forces on the right-hand
side of~\eqref{deriv-a} and~\eqref{deriv-b} to obtain
\begin{align*}
  \dd x_t^i &= v_t^i\dd t + \frac 1 n \sum_{j=1}^n F^x(x^{i}_t, v^{i}_t,x^j_t)\dd t, \label{ext_deriv-a}\\
  \dd v_t^i &= -\frac 1 n \sum_{j=1}^n \nabla H(v_t^i - v_t^j) \dd t +  \frac 1 n \sum_{j=1}^n F^v(x^{i}_t, v^{i}_t,x^j_t) \dd t +\dd W_t^i, 
\end{align*}
leading to the phase field SPDE
\begin{equation*}\dd \mu_t =  -\nabla_x \cdot (v \mu_t) \dd t +  \nabla_v\cdot ( \nabla H \ast \nu_t(x)\mu_t) \dd t  - \nabla \cdot( (F*\nu_t) \mu_t) + \frac 1 2 \Delta_v \mu_t \dd t
  +\alpha^ {-\frac 1 2 }\nabla_v \cdot\left(\sqrt{\mu_t}\dd W_{x,v,t}^v\right),
\end{equation*}
where $\alpha =n, F=(F^x, F^y)$, $\nabla \cdot = (\nabla_x, \nabla_v)\cdot$ is the full divergence with
respect to both coordinates $x$ and $v$, $\nu$ is the $x$-marginal of $\mu$ and $F*\nu$ is the convolution of
$F$ with $\nu$ with respect to the $x$ coordinate. This equation can be rescaled as above, resulting in the
prefactor of the stochastic perturbation becoming $1$ and the prefactor of the deterministic part becoming
$\alpha$.

Secondly, we note again that the derivation does not involve any heuristic steps or assumes the system to be
in local equilibrium, in contrast to previous
work~\cite{Cornalba2021a,Cornalba2019a,Cornalba2020a,Lutsko2012a}.

Thirdly, as already emphasized, the derivation shows that equations of fluctuating hydrodynamics are an exact
representation of the underlying particle system.  This seems to have been unclear, for two reasons: The
physical derivation of the classical Dean-Kawasaki equation~\cite{Dean1996a} considers a density $\rho_t$
associated to the underlying particle system and uses It\^o calculus to derive an evolution equation for this
density. The equation obtained is initially not closed, as the noise depends on the individual particles and
not on $\rho$. The noise is then replaced by a statistically equivalent one, which depends only on
$\rho_t$. It is \emph{a priori} not clear if this replacement is exact. The derivation above shows that it
indeed exact in the sense of It\^o/martingale calculus.

Also from a mathematical perspective there was possible confusion whether equations of fluctuating
hydrodynamics are correct or only an approximation. The reason is that the limit of particle systems, for
example Brownian particles, converge under suitable assumptions in suitable scaling to the hydrodynamic limit
as the particle number $n$ tends to infinity. For Brownian particles, the hydrodynamic limit, describing the
evolution of the limit measure $\bar\mu_t$, is $\partial_t \bar\mu_t = \Delta \bar\mu_t$. Then one can study
fluctuations around this limit, and the fluctuations are then described by
$\nabla \cdot\left(\sqrt{\bar\mu_t} W_{x,t}\right)$~\cite[Chapter~11]{Kipnis1999a}. Sometimes fluctuating
hydrodynamics is interpreted in form of an analogy, having the same form except with the coefficient
$\bar\mu_t$, i.e., the limit density, being described by $\mu_t$, i.e., the measure describing finitely many
particles. The calculation above shows this is not just a mere analogy, but indeed the correct description.

\bigskip 

\section{ Main result: Well-posedness and rigidity of Vlasov-Fokker-Planck Dean-Kawasaki type equations} 

We consider equations of the kind
\begin{align}
    \partial_t \mu_t = \alpha L^* \mu_t + \alpha \nabla \cdot(\mu_t \F
    _{\mu_t}) +  \nabla \cdot(\sqrt{\mu_t} \sigma \dot{W}_{z,t}), \label{eqn:DKE1}
\end{align}
where  $L^*$ is the dual of the operator $L$ given by
\begin{equation}
  \label{eq:L}
  L = b \cdot \nabla + \frac{1}{2}\sigma \sigma^T :\nabla^2, 
\end{equation}
with $b\colon\R^k \mapsto \R^k$ and $\sigma\colon \R^k \mapsto \R^{k\times l}$. In~\eqref{eqn:DKE1}
$\alpha >0$ is constant and the equation is to be understood as a formal It\^o SPDE driven by a
$l$-dimensional white noise vector field $\dot W_{z,t}$ on $\R^k$ (If domain and target are of the same
dimension, we omit the superscript in the notation). Here $(\mu,z) \mapsto \F(\mu,z) =: F_\mu(z)\in \R^d$
describes the interaction between the particles, where $z \in \R^k$. The divergence $\nabla\cdot$ is the full
divergence in $\R^k$.
 
For the existence result of Theorem~\ref{theo:main}, the force $F_\mu$ is assumed to be of the shape
$\sigma \sigma^T\nabla \frac{\delta G}{\delta \mu}$ for some functional $G$ cf.\
Assupmtion~\ref{assumption:interaction}. We remark that neither for the derivation nor for the existence
result the prefactor of $F$ matters; the results would hold if the second $\alpha$ in~\eqref{eqn:DKE1} is
replaced by a factor $\beta>0$; this is since $\alpha F = \beta \tilde F$ with
$\tilde F := \frac \alpha \beta F$ and the result holds for the forcing $\tilde F$ as well.

Equation~\eqref{eqn:DKE1} can be interpreted as follows: The operator $L$ consists, in the spirit of
GENERIC~\cite{Ottinger2005a}, of a conservative contribution $\nabla \cdot (b \mu)$ and a dissipative term
$\frac{1}{2}\sigma \sigma^T :\nabla^2$.  If $\F \neq 0$, then in the context of Theorem~\ref{theo:main}, both
dissipative term are balanced with the noise term in the sense that a fluctuation-dissipation theorem
holds. While this balance is not required for the derivation of the SPDE from the underlying particle model,
the proof of the (non-)existence dichotomy relies on this balance, as it is employed by the Girsanov transform
used in Section~\ref{sec:pf_for_f_ne_zero}.

Sometimes we write $z = (x,v)$ to indicate that $z$ can consist of components for space and velocity. For
example, let the state space consisting of positions $x \in \R^d$ and velocities $v\in \R^d$. We then set
$k = 2d$ and define $\sigma$ such that it vanishes for the first $d$ components (hypoelliptic setting) and is
the identity on the second $d$ components. Then~\eqref{intro.VLPDK} and~\eqref{eq:protoeqalpha} are examples
covered by~\eqref{eqn:DKE1}.

For a precise statement of our well-posedness result we extend the notion of weak solutions
from~\cite{Konarovskyi2019a} to the context of general, possibly degenerate and not necessarily reversible,
diffusions on $\R^k$. Here the possible degeneracy of the diffusion means that the noise does not have to act
on all variables. Indeed, for second order Langevin equations, the noise acts via a transport term only on the
velocities, but not on the positions. Such hypoelliptic situations as \eqref{intro.VLPDK} are therefore
covered by the analysis of this paper, under suitable assumptions on growth and regularity of the data.

Our analysis is based on the concept of (weak) martingale solutions in the space $\mathcal{M}_1(\R^k)$, the
space of probability measures on $\R^k$, equipped with the topology of weak convergence.

\begin{definition}[Martingale solutions]
  \label{def:martsol}
  A continuous $\mathcal M_1(\R^k)$-valued process $(\mu_t)_{t \geq 0}$ with $\mu_0 \in \cM_1$ is a \emph{solution} to
  Equation~\eqref{eqn:DKE1} if for each $\varphi \in C_b^2 (\R^k)$
  \begin{align*}
	M_t(\varphi) :&=  \av{\mu_t, \varphi} - \int_0^t 
                        \av{\mu_s, \alpha L \varphi + \alpha F_{\mu_s} \cdot \nabla \varphi} \dd s
  \end{align*}
  is a martingale with quadratic variation
  \begin{equation*}
    [M_\cdot(\varphi)]_t = \int_0^t \av{\mu_s, |\sigma^T \nabla \varphi|^2} \dd s.
  \end{equation*}
  \label{def:solutions}
\end{definition}
Above, $\langle \mu, g\rangle$ denotes the dual pairing of measures and functions, and the expression
$([M]_t)_{t \geq 0}$ stands for the quadratic variation of a general continuous (semi-)martingale $M$.

To establish well-posedness and rigidity of solutions, we impose the following conditions on $L$, which look
technical at first sight.  (For conditions on $\F$ see Assumption~\ref{assumption:interaction} below.) They
encode in weak but sufficient form certain well-posedness and regularity statements for the generalized
solutions to the underlying family of stochastic differential equations
\begin{equation}
     \dd z_t= b(z_t) \dd t+ \sigma(z_t) \dd W_{t},\quad  z_0=z,\label{eq:underlying_SDE} 
\end{equation}
where the relation between the differential operator $L$ and the coefficients $b$ and $\sigma$ is as
in~\eqref{eq:L}.  The advantage of the conditions given below is that they can be verified even for
non-regular data $b$ and $\sigma$, which occurs in some physically relevant models. In such cases the abstract
mathematical theory of Markov processes and their semigroups is arguably the most suitable framework for a
concise treatment. However, we chose to avoid the latter for the sake of a more self-contained
presentation. Instead we use the flexible notion of \textit{martingale solutions} to an SDE.\footnote{This is
  is a standard concept, see, e.g., the classical monographs~\cite{Ethier1986a,Stroock2006a}. We recall the
  basic definition for the reader's convenience. Given a class of test-functions $\varphi \in \mathcal D$ on
  $\R^d$, we call a stochastic process $(X_t)$ with values in $\R^d$ defined on some probability space
  $(\Omega, \mathcal F, \mathbb P)$ a solution to \eqref{eq:underlying_SDE} in the sense of the associated
  $\mathcal D$-martingale problem if for all $\varphi \in \mathcal D$
  \[ \left(M_t^\varphi\right)_{t\geq 0}:=\left(\varphi (X_t) - \varphi(x)- \int_{0}^{t} L \varphi(X_s)
      ds)\right)_{t\geq 0} \] is a real-valued martingale with $M_0^\varphi=0$, where $L$ is defined via the
  coefficients of \eqref{eq:underlying_SDE} as in \eqref{eq:L}. In short, we call such as process a
  \emph{solution to the $(L, \mathcal D, \delta_x)$-martingale problem}, where $\delta_x$ indicates the (in
  this case deterministic) initial distribution $X_0\sim \delta_x$. As an immediate consequence of It\^o's
  formula, one finds that any (weak or strong, understood in the probabilistic sense) solution of the
  SDE~\eqref{eq:underlying_SDE} is a solution of the $(L, \mathcal D, \delta_x)$-martingale problem for
  $\mathcal D= C_c^\infty(\R^d)$.  In case the coefficients of the SDE~\eqref{eq:underlying_SDE} are
  non-smooth, this equation might not be solvable in the It\^o sense, but solutions to the martingale problem
  might exits nevertheless. Martingale solutions are thus a very general concept of solutions.}
Assumption~\ref{assL:martp} below is a well posedness assumption on the law of the generalized solutions of
\eqref{eq:underlying_SDE}. (Mathematically, it means that there is a unique extension of the operator $L$
defined on observables to a Feller generator $\mathcal L$ satisfying, e.g., the classical conditions of the
Hille-Yosida theorem for generators of Markov contraction semigroups.) Condition \ref{assL:chainr} guarantees
that there is a rich enough class of test functions to which the It\^o formula can be applied twice, in a weak
form. Condition~\ref{assL:gradbound} encodes a mild form of stability with respect to the initial condition
$z \in \R^k$ of the family of SDE \eqref{eq:underlying_SDE}. Finally, condition~\ref{assL:exhaust} is
technical, but can be interpreted as a soft uniform condition on the diffusiveness of the process described
by~\eqref{eq:underlying_SDE}.

Throughout, we denote the space of bounded, twice continuously differentiable functions, mapping from $\R^k$
to $\R$, with bounded derivatives by $\C^2_b(\R^k)$.
\begin{assumption}
  \label{assumptionsonL}
  \begin{enumerate}[label=(L.{\arabic*})]
  \item \label{assL:martp}
    The operator $L$ acting on $\C_b^2(\R^k)$ admits a unique in law Markovian family of diffusion processes
    $X=(X_t^z)_{t \geq 0}^{z\in \R^k}$ solving the associated
    $(L,\C_b^2(\R^k), \delta_z)_{z \in \R^k}$-martingale problems.
       
  \item \label{assL:chainr} There is a set $\mathcal D\subset \C^2_b(\R^k) $ which is dense with respect to
    the topology of locally uniform convergence and is stable under composition with functions
    $\psi \in \C^\infty(\R)$ satisfying $\psi(0) = 0$ such that for all $\varphi \in \mathcal D$ and
    $t \geq 0$, it holds that $P_t \varphi \in \C^2_b(\R^k)$ and $P_t L\varphi = LP_t \varphi \in \C_b(\R^k)$,
    where $P_t \varphi(z) = \E \varphi(X_t^z)$.
  \item \label{assL:gradbound} For all $T>0$ and every function $\varphi \in \mathcal D$, the function
    $(z,t) \to \bigl|\sigma^T\nabla (P_t \varphi)\bigr|(z)$ is uniformly bounded on $\R^k\times [0,T]$.
  \item \label{assL:exhaust} There exists an exhaustion, i.e., a sequence of monotonically increasing sets $(A_n)_{n \in \N} \subset \mathcal{B}(\R^k)$ with
    $A_n \nearrow \R^k$, such that for any $t >0$ there exists a sequence $(c_n)_{n \in \N} \subset [0, 1)$
    such that $P_t 1_{A_n}(z) \leq c_n$ for all $z \in \R^k$ and $n \in \N$.
  \end{enumerate}
\end{assumption}

The unique solution $X$ of the $(L,\C_b^2(\R^k))$-martingale problem granted by condition~\ref{assL:martp}
will be called \emph{$L$-diffusion}. We point out that Assumption~\ref{assumptionsonL} can be substantially
relaxed in our main theorem if one uses the framework of semigroup theory. For instance, if $X$ is a Feller
diffusion, the analogue of condition~\ref{assL:chainr} reads
$P_t\mathcal L\varphi =\mathcal LP_t\varphi \in \C_b(\R^k)$
and is automatically satisfied for $\varphi \in \mathcal D =D(\mathcal L )$, the domain of the Feller
generator $\mathcal L$, which is dense. Standard cases when Assumptions~\ref{assL:martp}, \ref{assL:gradbound}
and~\ref{assL:exhaust} hold are when $b,\sigma$ are either smooth and of moderate growth or satisfy
ellipticity conditions or suitable hypoellepticity conditions. In Section~\ref{section:langeveinfzero} we show
in full detail that Assumption~\ref{assumptionsonL} holds in the classic Langevin case, where $k=2d$ with
$\sigma = (0_{\R^d},\mbox{Id}_{\R^d})^T$, and regular $b$ for $\mathcal D$ consisting of twice continuously
differentiable functions with uniformly continuous second derivative.

\begin{remark}
  \label{rem-carre}
  The operator $\sigma^T\nabla$ appearing in assumption~\ref{assumptionsonL} is related to the \emph{carr\'e
    du champ} operator $\Gamma(f,g) = \frac{1}{2}\left( L(fg) - fLg -gLf \right) $ associated with $L$,
  \begin{align*}
    \Gamma(f,g) = \frac 1 2 (\sigma^T \nabla f) \cdot \sigma^T \nabla g \quad \forall f,g \in \C_c^\infty(\R^k).
  \end{align*}
  For simplicity, we write $\Gamma(f) = \Gamma(f,f)$.
\end{remark}

\smallskip

The case of nonzero $\F$ is treated as perturbation of the case where $\F=0$. For the sake of a concise
argument, we will impose further, rather restrictive, assumptions on $L$ and $\F$ below, not aiming for
maximal generality. In Section~\ref{sec:examples}, however, it is shown that a wide range of physically
relevant models is covered. Apart from asking for more regularity on the coefficients of $L$, our basic
structural assumption on the interaction $\F$ is that it originates from an interaction potential and is of
sufficient regularity. To this aim, for a real valued function
$G\colon D(G) \subset \mathcal{M}(\R^k) \to \R$, $z \in \R^k$ and $\mu \in D(G)$, let
\begin{equation*}
  \frac{\delta G}{\delta \mu}(\mu, z) =  \partial_{\epsilon|\epsilon=0} G(\mu+\epsilon\delta_z),
\end{equation*} 
whenever this directional derivative exists; here $\mathcal M(\R^k)$ is the space of finite measures.  We
define $\frac{\delta^2 G}{\delta \mu^2}(z,w)$ etc.\ analogously.  Further, we write $\cM^2_1 \subset \cM_1$
for the set of probability measures with finite second moment.

In the interacting case $\F\neq0$, the following additional conditions are required.
\begin{assumption}
\label{assumption:interaction} 
 \begin{enumerate}[label=(L.{\arabic*})] \addtocounter{enumi}{4}
 \item \label{interact.L} The fields $\sigma$ respectively $b$ are continuous and bounded respectively of at
   most linear growth. That is, there is some $K>0$ such that $|\sigma(z)|\leq K$ and $|b(z)|\leq K(1+|z|)$
   for $z \in \R^k$.
\end{enumerate}
\begin{enumerate}[label=(F.{\arabic*})]
\item \label{F-F1} %
  For some $G \in \C^2_{loc}(\cM_1^2)$ as defined in Definition~\ref{def:c2b_etc}, we have
  \begin{equation*}
    F_\mu(\cdot) = \left(\sigma \sigma^T\nabla\frac {\delta G }{\delta \mu}\right)(\cdot).
  \end{equation*}
\item \label{F-F2} It holds that $\sigma^ T\nabla \frac {\delta G}{\delta \mu}(\mu,.)$ is bounded and
  $\nabla^l \frac{\delta^m G}{\delta \mu^m}(\mu,.)$ is for
  $m,l\in \{1,2\}$ of at most linear growth, locally uniformly with respect to $\mu$ on $\cM_1$ and
  $\cM^2_1$ respectively.
\end{enumerate}
\end{assumption}
(For details of the conditions~\ref{F-F1} and~\ref{F-F2}, we refer to Definition~\ref{def:c2b_etc} in
Section~\ref{sec:pf_for_f_ne_zero} below.) We can now state our main result, extending~\cite{Konarovskyi2019a,
  Konarovskyi2020a} to the case of possibly degenerate and non-reversible diffusions.

\begin{theorem}
  \label{thm:mainthm}
  Let Assumption~\ref{assumptionsonL} hold, and additionally Assumption~\ref{assumption:interaction} be
  satisfied if the interaction $\F$ is nonvanishing, $\F\ne 0$. Let $\mu_0$ be a probability measure with
  finite second moment. Then the initial value problem associated with~\eqref{eqn:DKE1} and initial condition
  $\mu_0$ has a solution if and only if $\alpha = n$ for some $n\in \N$.  In this case, the solution is given
  by
  \begin{equation*}
    \mu_t = \frac{1}{n} \sum_{i =1}^n \delta_{z_ t^i}, 
  \end{equation*}
 where the system $ \{z_t^i\}_{t\geq 0}^{i=1\cdots, n}$ is the unique in law solution to
  \begin{equation*}
    \dd z_t^i = \alpha \, b(z^i_t) \dd t  + \alpha \F\left(\frac 1 n
      \sum_{j=1}^n\delta_{z_t^{j}},z^i_t\right) \dd t + \sqrt {\alpha}\, \sigma(z^i_t) \dd W^i_t , \quad i = 1, \cdots ,n,
  \end{equation*}
  driven by $n$ independent $\R^l$-valued Brownian motions $\{(W^i_t)_{t\geq 0}\}^{i= 1, \cdots, n}$.
  \label{theo:main}
\end{theorem}
We make two remarks. First, as we start with the SPDE~\eqref{eqn:DKE1} rather than with the interacting
particle systems as in the examples above, we now have to rescale the underlying particle process to match the
scaling of the SPDE.  Second, we remark that the result above could have been stated for the counting measure
$\hat\mu_t(x,v) = \sum_{i=1}^n \delta_{(x_t^i, v_t^i)}$ as studied by Dean~\cite{Dean1996a}, allowing for
integer-valued initial conditions.

Furthermore, similar to the discussion in~\cite{Konarovskyi2020a, Konarovskyi2024a} for the overdamped case,
the situation of general non-negative total mass can be considered for second order dynamics. In this case,
the restriction that the scaling $\alpha$ be integer-valued is modified: Denote the total mass by
$m= \mu_0(\R^k)$, then solutions exists if and only if $\alpha m \in \mathbb \N $ and are of the form
$\frac{1}{\alpha} \sum_{i = 1}^{\alpha m} \delta_{z_{\alpha t}^i}$.
Before giving the proof of Theorem~\ref{thm:mainthm}, we present a collection of examples where the result
applies.

\section{Examples}
\label{sec:examples}

\subsection{Inertial Langevin dynamics without interaction}
\label{section:langeveinfzero}

The motivation of this work is inertial Langevin dynamics, i.e., $n$ particles which have at time $t \geq 0$ a
position $x_t$ and a velocity $v_t$, evolving by
\begin{align}
  \begin{split}
    \dd x_t &=  v_t \dd t,\\
    \dd v_t &=  - \gamma v_t \dd t - \nabla U(x_t) \dd t + \dd W_t,
  \end{split}
  \label{eqn:inertial}
\end{align}
where $W_t$ is a $d$-dimensional Brownian motion and $(x_t,v_t) \in \R^d \times \R^d$. We assume that
$\gamma \geq 0$, that $U \in \C^3(\R^d)$ is uniformly bounded from below, and that $\nabla^2 U$ and
$\nabla^3 U$ are bounded. To fit Equation~\eqref{eqn:inertial} into the framework of
Equation~\eqref{eqn:DKE1}, we set $k = 2d$ and write $z \in \R^{k}$ as a combined vector of position and
velocity, $z = (x,v) \in \R^d \times \R^d$. Then the drift vector $b$ and the multiplicative noise $\sigma$
in~\eqref{eqn:DKE1} are given by
\begin{align*}
  b(x,v) = \begin{pmatrix}
    v \\ -\gamma v - \nabla U(x)
  \end{pmatrix}\quad
  \text{and} \quad
  \sigma = \begin{pmatrix} 0_{d \times d} \\ 1_{d \times d}\end{pmatrix},
\end{align*}
where $0_{d \times d}$ is a $d \times d$ matrix of zeros and $1_{d\times d}$ is the identity matrix in $d$
dimensions. This defines the operator $L$ of~\eqref{eq:L}.

In this case the corresponding degenerate Dean-Kawasaki Equation (Vlasov-Fokker-Planck equation) reads, with
$F = 0$, after rescaling in time as described in Section~\ref{sec:deriv-equat-fluct}, with $\alpha=n$
\begin{equation*}
  \partial_t \mu_t =  \alpha\left(\frac 1 2\Delta_v \mu_t -  \nabla_x \cdot (v \mu_t) + \nabla_v
  \cdot( (\gamma v + \nabla U(x)) \mu_t) \right) + \nabla_v \cdot (\sqrt{\mu_t} \dot W_{x,v,t}^v).
\end{equation*}

Theorem~\ref{thm:mainthm} asserts that the unique solution is of the form
$\frac 1 n \sum_{i=1}^n\delta_{(x_{n t}^i,v_{ n t}^i)}$, composed of independent solutions
$\{(x_t^i,v_t^i)_{t \geq 0}^{i=1,\cdots, n}\}$ of~\eqref{eqn:inertial}. To apply the theorem, we have show
that Assumption~\ref{assumptionsonL} on $L$ is fulfilled.

\emph{Assumption~\ref{assL:martp}:} The conditions imposed on $U$ imply that $b$ is in
$\C^1(\R^d \times \R^d, \R^d \times \R^d)$ with bounded first derivatives. Thus $b$ is globally Lipschitz
continuous. This guarantees the existence and uniqueness of a strong solution to the
SDE~\eqref{eqn:inertial}. Furthermore, for any $\varphi \in \C^2_b(\R^{d})$, by It\^o's formula, the process
\begin{align*}
  M_t (\varphi) =& \int_0^t \nabla_v \varphi(x_s, v_s) \cdot \dd W_s
                =\varphi(x_t,v_t) - \varphi(x_0,v_0) - \int_0^t L\varphi(x_s,v_s) \dd s 
\end{align*}
is a local martingale, for $t\geq 0$. Further, $M_\cdot(\varphi)$ is a square integrable martingale, since
\begin{equation*}
  [M_\cdot(\varphi)]_t = \int_0^t |\nabla_v \varphi(x_s, v_s)|^2 \dd s \leq ||\nabla_v
  \varphi||_\infty^2 t .
\end{equation*}
Thereby, $(x_t,v_t)_{t \geq 0}$ uniquely solves the $(L, \C^2_b(\R^k))$-martingale problem and
Assumption~\ref{assL:martp} is satisfied.

\smallskip

\emph{Assumptions~\ref{assL:chainr} and~\ref{assL:gradbound}:} We choose $\mathcal{D}$ to be the space of
functions $\varphi\colon \R^{2d} \to \R$ which are uniformly continuous and bounded together with their
derivatives up to second order.  Due to the bounds on $\nabla^2 U$ and $\nabla^3 U$, the vector field $b$ is
Lipschitz continuous and the second derivatives of $b$ are bounded. The system thus satisfies the assumptions
of~\cite[Proposition~2.5]{Da-Prato2011a}. According to this result, $P_t \mathcal D \subset \mathcal D$ and
there exists some $K > 0$ such that $||P_t \varphi||_{2, \infty} \leq K||\varphi||_{2,\infty}$ for all
$t \geq 0$ and $\varphi \in \mathcal D$.  Here (and later),
$||\varphi||_{2,\infty} = ||\varphi||_\infty + ||\nabla \varphi||_\infty + ||\nabla^2 \varphi||_\infty$ for
functions in $\mathcal D$.  This already implies Assumption~\ref{assL:gradbound}, since $\sigma$ is bounded.
Furthermore, Propositions~3.3 and~3.6 of~\cite{Da-Prato2011a} together state that
$L P_t \varphi = P_tL \varphi$ and that $L \varphi$ is a uniformly continuous bounded function, thus in
particular $P_t L \varphi \in \C_b$. Hence Assumption~\ref{assL:chainr} is satisfied as well.

\smallskip

\emph{Assumption~\ref{assL:exhaust}:} To establish this property, we first note that our assumptions on $U$
imply that the semigroup induced by the solutions to~\eqref{eqn:inertial} is Feller and, moreover,
by~\cite[Proposition 1.2]{Wu2001b} strong Feller.  Hence $P_t 1_A \in \C_0\left(\R^{2d}\right)$, if
$A\subset \R^{2d}$ is bounded and measurable.

In particular, for given $r>0$, there exists a $R > 0$ such that $P_t 1_{B(0,r)}(z) \leq \frac{1}{2}$ for all
$|z| \geq R$. It thus remains to estimate $\max_{z \in\overline{B(0,R)}} P_t 1_{B(0,r)}(z) $.  Let
$z^* \in \R^{2d}$ a point at which the maximum is attained.  Now, by Girsanov's theorem
\begin{gather*}
  \P\left( |Z_t^{z^*}| > r \right) = \int M_t 1_{ \{|Z_t^{z^*,0}|> r\}} \dd P_{z^*}^0 > 0,
  \intertext{where}
  M_t =  \exp\left( - \frac{1}{\sqrt{2}} \int_0^t (\gamma v_s + \nabla U(x_s)) \dd W_s
    + \frac{1}{4} \int_0^t \left|\gamma v_s + \nabla U(x_s)\right|^2 \dd s \right),
\end{gather*}
and $P^0_{z^*}$ is the law of $(Z_t^{{z^*},0}=(x^*+\int_0^t W_sds, v^*+W_t) )_{t\geq 0}$
solving~\eqref{eqn:inertial} with $\gamma=0$ and $U=0$ starting from ${z^*}=(x^*,v^*)$,
see~\cite{Wu2001b}. Since there is no explosions (see~\cite[Lemma 1.1]{Wu2001b}), it holds that $M_t > 0$
almost surely. Moreover $\{|Z_t^{z^*,0}| > r\} \supset \{ |v^* + \sqrt{2}|W_t| > r\}$ is not a nullset. Hence,
$\P\left( |Z_t^{z^*}| > r \right) >0$, so that indeed
\begin{equation*}
  {\sup}_{z \in \R^{2d}} \displaystyle P_t 1_{B(0,r)}(z) \leq \max\left( \max_{z \in
      \overline{B}(0,R)} P_t 1_{B(0,r)}(z), \frac{1}{2} \right) < 1,
\end{equation*}
which yields Assumption~\ref{assL:exhaust} by choosing, for example,\ $A_n=B(0,n)\subset \R^{2d}, n \in \N$.

\subsection{Active matter}
\label{sec.active-matter}

We now consider a class of active matter models, where $n$ individual agents at position $x_t \in \R^d$ are
driven by a propulsion mechanism $g\colon \R^l \to \R^d$, for some $l \leq d$,
\begin{align}
    \begin{split}
        \dd x_t &= g(\theta_t) \dd t,\\
    \dd \theta_t &= \dd W_t.
    \end{split} \label{eqn:active}
\end{align}
In this case, the corresponding degenerate Dean-Kawasaki equation (Vlasov-Fokker-Planck equation) reads after
rescaling in time, $t \mapsto \alpha t$ with $\alpha = n$,
\begin{equation*}
  \partial_t \mu_t = \alpha\left(\frac{1}{2} \Delta_\theta \mu_t -  \nabla_x \cdot (g(\theta) \mu_t)\right)
  + \nabla_\theta \cdot (\sqrt{\mu_t} \dot W_{x,\theta,t}^\theta).
\end{equation*}
The classical example is $g(\theta) = v \begin{pmatrix} \cos(\theta) \\ \sin(\theta)
\end{pmatrix}$ for some $v \in \R$. Then the angle $\theta$ can be interpreted as the direction the active
particle is propelled into. See, for example,~\cite{Peruani2008a,Bauerle2018a} and references therein. We
assume that $g \in \C^2_b\left(\R^l, \R^d\right)$.  Then, a strong solution to the SDE~\eqref{eqn:active}
exist; since the coefficients are Lipschitz continuous, the solution is unique. Hence
$(x_t, \theta_t)_{t \geq 0}$ solves the $\left(L, \C^2_b\left(\R^{d+l}\right)\right)$ martingale problem for
the generator
\begin{align*}
    L \varphi(x,\theta) = g(\theta) \cdot \nabla_x \varphi(x,\theta) + \frac{1}{2}\Delta_\theta \varphi(x,\theta).
\end{align*}
With the same techniques as in Section~\ref{section:langeveinfzero}, one can show that
Assumptions~\ref{assL:chainr} and~\ref{assL:gradbound} hold.  Again $\mathcal D$ consists of the space of
functions $\R^{d +l} \to \R$ which are uniformly continuous and bounded together with their first and second
derivatives.
Finally, to see that~\ref{assL:exhaust} is satisfied, we set $A_n := \R^d \times B(0,n)$. Then,
$P_t 1_{A_n}(x,\theta) \leq \sup_{\theta \in \R^k} \P(W_t + \theta \in B(0,n)) < 1$ for any $t > 0$.

\subsection{Fluctuating hydrodynamics for Vlasov-Fokker-Planck equation with interaction}
\label{section:langeveinfnotzero}

We now extend the interaction-free example of Section~\ref{section:langeveinfzero} to the case of second order
Langevin particles with interaction, under the same conditions on the on-site potential $U$ as in
Section~\ref{section:langeveinfzero}. In addition, we assume that particles experience an interaction
$f_\text{int} \in C^2_b(\R^d \times \R^d, \R^d)$, which is a function of the positions. That is, $F_\mu$
depends on the spatial marginal $\rho(\cdot)=\mu(\cdot\times \R^d)$ via
\begin{align*}
    F_{\mu}(x) = \int_{\R^d} f (x,x') \dd \nu(x') = f \ast\nu(x)
\end{align*}
(with a slight abuse of notation of the convolution $\ast$ in case the pair interaction does not only depend
on the difference of particle positions).
In this case, the corresponding Vlasov-Fokker-Planck equation (degenerate Dean-Kawasaki equation) is obtained
as described in Section~\ref{sec:deriv-equat-fluct}, with rescaling in time $t \mapsto \alpha t$, where
$\alpha = n$ is the number of particles,
\begin{align*}
  \partial_t \mu_t = \alpha\left(\frac{1}{2} \Delta_v \mu_t -  \nabla_x \cdot (v \mu_t)
  + \nabla_v \cdot( (\gamma v + \nabla U(x)) \mu_t + \mu_t (f * \rho_t))\right) + \nabla_v \cdot (\sqrt{\mu_t} \dot W_{x,v,t}^v),
\end{align*}
and the corresponding particle system $\left(x_t^i, v_t^i\right)_{i =1}^n$ solves for $i = 1, \dots, n$
\begin{align*}              
\dd x_t^i &=  v_t^i \dd t,\\
    \dd v_t^i &=  - \gamma v_t^i \dd t - \nabla U(x_t^i) \dd t - \frac 1 n  \sum_{j =1}^n f(x_t^i,x_t^j) \dd t  +  \dd W_t^i.
\end{align*}
One checks that condition~\ref{F-F1} of Assumption~\ref{assumption:interaction} is verified by choosing
\begin{equation*}
  G(\mu) = \av{\mu, v \cdot F_\mu},
\end{equation*}
so that indeed 
\begin{equation*}
  \nabla_v \frac{\delta G}{\delta \mu}(x,v) = \nabla_v\left(v \cdot F_\mu(x) + \int_{\R^{2d}} v'\cdot f (x,x') \dd \mu(x',v')\right)
  = F_{\mu}(x).
\end{equation*}
Moreover, $G$ satisfies the regularity and growth condition~\ref{F-F2}.

\subsection{Active swimmmers and flocking}
\label{sec.swimmers-flocking}

We expand on the example of active matter in Section~\ref{sec.active-matter} to include flocking behaviour. In
the situation where the individual agents occupy a position $x_t \in \R^d$ and are driven by a propulsion
mechanism $g\colon \R^l \to \R^d$, for some $d \geq k$, depending on a noise $\theta_t$ as specified in
Equation~\eqref{eqn:active}, flocking behavior is induced by
interactions of the type
\begin{align}
  F_\mu(x, \theta) = \int \chi_R(|x-x'|) \nabla H(\theta-\theta') \dd \mu(x',\theta'),
  \label{eqn:activeinteraction}
\end{align}
where for some interaction radius $R >0$, $1_{B_R} \leq \chi_R \leq 1_{B_{R+\epsilon}}$ is some smooth cut-off
function and $H \in \C^1_b(\R^k)$. This means that active swimmers interact with those in their
neighbourhood. A classical example is $H(\theta) = -\cos(\theta)$, that is particles tend to align.

In this case the corresponding equation of fluctuating hydrodynamics  reads
\begin{equation*}
  \partial_t \mu_t = \alpha\left(\frac{1}{2} \Delta_\theta \mu_t -  \nabla_x
    \cdot (g(\theta) \mu_t) + \nabla_\theta \cdot (\mu_t F_{\mu_t})\right) +\nabla_\theta \cdot (\sqrt{\mu_t} \dot W_{x,\theta,t}^\theta),
\end{equation*}
and the corresponding particle system solves
\begin{align*}
    \dd x_t^i &= g(\theta_t^i)\dd t,\\
    \dd \theta_t^i &= -\sum_{j \sim i} \chi_R(|x-x'|) \nabla H( \theta_t^i - \theta_t^j) \dd t + \dd  W_t^i,
\end{align*}
where indices $i, j$ satisfy $i \sim j$ iff $|x_t^i -x_t^j| < R+1$ and $W_\cdot^1, \dots, W_\cdot^n$ are
independent Brownian motions.

In view of Section~\ref{sec.active-matter}, to apply our main result, Theorem~\ref{thm:mainthm}, it remains to
check condition~\ref{F-F2} of Assumption~\ref{assumption:interaction} is satisfied. Indeed, this is the case
if we assume that $H$ belongs to $\C^2_b$ and is even, and choose
\begin{equation*}
  G(\mu) = \frac{1}{2}\int\int \chi_R(|x-x'|) H(\theta - \theta')\dd \mu(x,\theta) \dd \mu(x',\theta').
\end{equation*}

\section{Proof of the main theorem}
\label{sec.proof}

The proof proceeds in two steps, extending the the basic strategy of of~\cite{Konarovskyi2019a,
  Konarovskyi2020a} for symmetric Markov operators to the present situation of possibly degenerate diffusions
without symmetry.  The central tool for the case $F=0$ is a duality (Proposition~\ref{prop:laplace}) involving
the Laplace transform of the measure valued random variable $\mu_t$. As a result, we obtain a representation
of its moment generating function in terms of a deterministic Hamilton-Jacobi-Bellman equation. This allows to
establish regularity of the moment generating function near the origin, which yields the rigidity of the model
when $F=0$ (Proposition~\ref{lemma:nonexistence}). In a second step the case $F\ne 0$ is reduced to the
situation when $F=0$ via an appropriate Girsanov transform (Proposition~\ref{prop:girsanovfordke}). As main
difference to the previous works~\cite{Konarovskyi2019a, Konarovskyi2020a}, some additional care is needed due
to the reduced regularity of the associated (possibly degenerate) nonlinear Hamilton-Jacobi-Bellman semigroup
in the proof of the duality formula. In order to treat physically interesting models with with unbounded
drift, we establish the stability of the second moment of $\mu_t$ as presented in
Lemma~\ref{lemma:secondmoment}. In particular, the rigidity statements of the previous works op.\ cit.\ are
extended to a wider class of interactions.

\subsection{Interaction-free case}
\label{sec.Fnull}

We now prove Theorem~\ref{theo:main} for the noninteracting case, $\F=0$. We restate the result for this case.

\begin{proposition}
  \label{prop:DK-L-allg}
  Under Assumption~\ref{assumptionsonL}, a solution of the SPDE 
  \begin{align}
    \partial_t \mu_t &= \alpha L^* 
                       \mu_t + \nabla_z \cdot \left( \sqrt{\mu_t} \sigma \dot{W}_{z,t} \right)
                       \label{eqn:F0} \\
    \intertext{with initial data} 
    \mu_0 &= \mu \label{eqn:F0-IC}
  \end{align}
  and $\alpha>0$ exists if and only if $\alpha \in \N$ and $\mu_0$ is of the form
  $ \mu_0= \frac{1}{\alpha } \sum_{i =1}^\alpha \delta_{z^i}$ for some $z^1, \dots, z^\alpha \in \R^k$. In
  this case the solution is given by
  \begin{equation*}
    \mu_t = \frac 1 \alpha \sum_{i=1}^\alpha \delta_{z^i_t},
  \end{equation*}
  where $\{z^i_\cdot\}_{i=1}^N$ are independent $\alpha L$-diffusions.
\end{proposition}

We recall that solutions are understood in the martingale sense, as stated in Definition~\ref{def:solutions}.
We first show that solutions in the martingale sense exists for initial data of the form
$\mu = \frac{1}{\alpha} \sum_{i =1}^N \delta_{z^i}$, and defer the proof of the uniqueness and the rigidity to
Propositions~\ref{lemma:uniqueness} and~\ref{lemma:nonexistence} respectively.  The existence of such
solutions with atomic initial data essentially follows from Assumption~\ref{assumptionsonL},
part~\ref{assL:martp} and can be seen as rigorous extension of the derivation of the equations of fluctuating
hydrodynamics in Section~\ref{sec:deriv-equat-fluct}. Indeed, let us choose $N$ independent solutions
$(z_\cdot^i)_{i =1}^N$ to the $\left(\alpha L, \C^2_b\left(\R^k\right)\right)$-martingale problem and set
$\mu_t^N = \frac{1}{\alpha} \sum_{i =1}^N\delta_{z_t^i}$. Then the process
\begin{align*}
  M_t(\varphi) :&=  \av{\mu_t^N, \varphi} - \av{\mu_0, \varphi} - \int_0^t  \av{\mu_s^N, \alpha L \varphi } \dd s\\
                & = \frac{1}{\alpha}\sum_{i = 1}^N \left(\varphi(z_t^i) -\varphi(z_0^i) - \int_0^t \alpha L \varphi(z_s^i) \dd s \right) 
\end{align*}
is a sum of martingales and hence a martingale.  Note that the rescaling by $\alpha$ corresponds to a
multiplication of the diffusion coefficient $\sigma$ by $\sqrt{\alpha}$.  (In fact, from the definition of a
martingale solution, it follows immediately that if $\mu_t$ is a solution of~\eqref{eqn:F0-IC}, then
$\tilde \mu_t=\mu_{t/\alpha}$ solves
\[   \partial_t \tilde\mu_t =  L^* \tilde \mu_t 
                       +\frac 1 {\sqrt \alpha}   \nabla_z \cdot \left( \sqrt{\tilde\mu_t} \sigma \dot{W}_{z,t} \right)\]
with initial the same initial condition.) Due to the independence of the processes, the
quadratic variation of the aforementioned martingale is given by
\begin{align*}
  \frac{1}{\alpha^2}\sum_{i = 1}^N |\sqrt{\alpha}\sigma^T \nabla \varphi(z_s^i)|^2 \dd s
  = \int_0^t \av{\mu_s^N, |\sigma^T \nabla \varphi|^2} \dd s.
\end{align*}
This shows existence of a martingale solution in the sense Definition~\ref{def:solutions}, for initial
conditions given by $\frac{1}{\alpha}\sum_{i =1}^N \delta_{z^i_t}$.

\smallskip

Uniqueness in law follows from a Laplace duality, as in~\cite{Konarovskyi2019a}, and the proof is is similar
with small modifications. Namely, we show in Proposition~\ref{prop:laplace} that
\begin{equation*}
    \E \left(e^{-\av{\mu_t,\varphi}}\right) = e^{-  \av{\mu_0, V_t \varphi}}
\end{equation*}
holds, where $ V_t \varphi$ is a solution of the Hamilton-Jacobi-Bellman equation
\begin{align}
	\partial_t \psi_t &= \alpha L\psi_t - \Gamma(\psi_t)
 \intertext{with initial condition}
	\psi_0 & =\varphi \in \mathcal D,
    \label{eq:HJ}
\end{align}
where $\mathcal D$ is chosen as in~\ref{assL:chainr}. This equation admits the Cole-Hopf solution
$\psi_t = V_t \varphi := - \alpha \ln\left(P_{\alpha t} e^{-\frac{\varphi}{\alpha}}\right)$.  Indeed, since
$f(y) = e^{-y} -1$ is in $\C^\infty(\R)$ with $f(0) = 0$, we have that
$\zeta := e^{- \frac{\varphi}{\alpha}} -1 \in \mathcal D$ and so
$P_t e^{-\frac{\varphi}{\alpha}} = P_t \zeta +1 \in \C^2_b(\R^k)$ and $\partial_t P_t \zeta = LP_t \zeta$,
both by Assumption~\ref{assL:chainr}.
Hence, we can use the standard chain rule of calculus to compute
\begin{align*}
	L V_t \varphi = -\alpha L \left(\ln(P_{\alpha t} e^{-\frac{\varphi}{\alpha}}) \right) &
	= - \alpha \frac{1}{P_{\alpha t} e^{-\frac{\varphi}{\alpha}}} L P_{\alpha t} e^{-\frac{\varphi}{\alpha}} 
	+ \alpha \left(\frac{1}{P_{\alpha t} e^{-\frac{\varphi}{\alpha}}}\right)^2 \frac 1 2 \left| \sigma^T \nabla P_{\alpha t} e^{-\frac{\varphi}{\alpha}}\right|^2\\
	&=  - \alpha \frac{1}{P_{\alpha t} e^{-\frac{\varphi}{\alpha}}} L P_{\alpha t} e^{-\frac{\varphi}{\alpha}}
	+  \frac{1}{\alpha} \Gamma(V_t \varphi) \intertext{with the carr\'e du champs operator from Remark~\ref{rem-carre}, and also} 
	\partial_t V_t \varphi &=  \frac{-\alpha}{P_{\alpha t} e^{-\frac{\varphi}{\alpha}}} \alpha L P_{\alpha t} e^{-\frac{\varphi}{\alpha}},
 \end{align*}
showing that $V_t\varphi$ is indeed as solution. 

Moreover, for the Cole-Hopf solution it holds that
\begin{align*}
e^{-\frac{\sup_{x \in \R^k} \varphi(x)}{\alpha}} &= \inf_{x \in \R^k} e^{-\frac{\varphi}{\alpha} } \leq P_t e^{-\frac{\varphi}{\alpha}} \leq \sup_{x \in \R^k} e^{-\frac{\varphi}{\alpha} } = e^{-\frac{\inf_{x \in \R^k} \varphi}{\alpha}},
\end{align*}
which leads to 
\begin{equation*}
  ||V_t \varphi||_\infty \leq ||\varphi||_\infty \quad \forall t \geq 0.
\end{equation*}

\smallskip

\begin{proposition}[Laplace duality]
  \label{prop:laplace}
  Assume $\alpha >0$ and $\mu_0 \in \mathcal M_1$. Then for all non-negative functions
  $\varphi \in \C_b(\R^k)$
  \begin{align}
    \E \left(e^{-\av{\mu_t,\varphi}}\right) = e^{-  \av{\mu_0, V_t \varphi}}.  
    \label{eqn:laplace}    
    \end{align}     
    Here, the expectation is taken with respect to the probability measure of the initial value
    problem~\eqref{eqn:F0}--\eqref{eqn:F0-IC}.
\end{proposition}

\begin{proof}
  First we assume that $\varphi \in \mathcal D$ with $\mathcal D$ given as in Assumption~\ref{assL:chainr},
  and $\varphi \geq 0$. Using approximations of time-dependent functions
  $\psi_\cdot\colon [0,\infty)\times \R^k \to \R$ of sufficient regularity with
  $\sup_{t \in [0,T]}||\psi_t||_\infty + ||L \psi_t||_\infty + ||\Gamma(\psi_t)||_\infty< \infty$ as
  in~\cite{Konarovskyi2024a}, a time-dependent version of the notion of martingale solutions given in
  Definition~\ref{def:solutions} may be established. Namely,
  \begin{equation*}
    M_t(\psi_\cdot):=  \av{\mu_t, \psi_t} - \av{\mu_0, \psi_0} - \int_0^t \av{\mu_s, L \psi_s + \partial_s \psi_s} \dd s
  \end{equation*}
  is a martingale with quadratic variation
  \begin{equation*}
    [M_\cdot]_t =  \int_0^t \av{\mu_s, |\sigma^T \nabla \psi_s|^2} \dd s.
  \end{equation*}
  Now fix $t > 0$ and define $\psi_s = V_{t-s} \varphi$.  Note that by the gradient bound~\ref{assL:gradbound}
  of Assumption~\ref{assumptionsonL}, the diffusion property and the fact that
  $e^{-\frac{\varphi}{\alpha}}-1 \in \mathcal{D}$, $(\psi_s)_{s \in [0,t]}$ indeed satisfies the required
  boundedness of $L V_{t-s} \varphi \in \C_b$ and $\Gamma(V_{t-s} \varphi)$.  Thus, we can apply the
  time-dependent It\^o formula and obtain that $\left(e^{-\av{\mu_s, V_{t-s}\varphi}}\right)_{s \in [0,t]}$ is
  a local martingale, since
  \begin{align*}
    \dd e^{-\av{\mu_s, V_{t-s}\varphi}} &=  -e^{-\av{\mu_s, V_{t-s}\varphi}} \dd \av{\mu_s, V_{t-s}\varphi}
                                          + \frac{1}{2} e^{-\av{\mu_s, V_{t-s}\varphi}} \dd [M_\cdot(V_{t-\cdot} \varphi)]_s\\
                                        &=   e^{-\av{\mu_s, V_{t-s}\varphi}} \biggl( -\av{\mu_s, \alpha L V_{t-s}\varphi
                                          + \partial_s V_{t-s} \varphi } \dd s \biggr. \\ &{}\qquad \biggl. - \dd M_s(V_{t-\cdot}\varphi) + \frac{1}{2} \av{\mu_s, |\sigma^T\nabla_v V_{t-s} \varphi|^2} \dd s \biggr)\\
		&= -e^{-\av{\mu_s, V_{t-s}\varphi}}\dd M_s(V_{t-s}\varphi) , 
  \end{align*}
  where the last line uses that $V_{t-s} \varphi$ is a solution to the Hamilton-Jacobi-Bellman
  equation~\eqref{eq:HJ}.  This local martingale is uniformly bounded, since for any non-negative compactly
  supported function $\varphi$, the Cole-Hopf transform satisfies $0 \leq V_t \varphi \leq ||\varphi||_\infty$
  and hence $0 \leq e^{\av{\mu_s, V_{t-s}\varphi}} \leq 1$.  At this point, we have established the desired
  Equation~\eqref{eqn:laplace} for all non-negative $\varphi \in \mathcal{D}$; by approximating functions
  pointwise from below and monotone convergence, we obtain the equality for all nonnegative
  $\varphi \in \C_b(\R^k)$.
\end{proof}

\begin{proposition}
  \label{lemma:uniqueness}
  Under Assumption~\ref{assumptionsonL}, solutions to the initial value
  problem~\eqref{eqn:F0}--\eqref{eqn:F0-IC} are unique in law.
\end{proposition}
 
\begin{proof}
  As in~\cite{Konarovskyi2019a}, uniqueness in law of the one-time marginals follows directly from the Laplace
  duality (Proposition~\ref{prop:laplace}): By monotone convergence, Equation~\eqref{eqn:laplace} extends to
  all non-negative bounded functions $\varphi$. For fixed $t >0$ and non-negative $\varphi \in \C_b(\R^k)$,
  the Laplace transform uniquely determines the distribution of the real-valued random variable
  $\av{\mu_t, \varphi}$.  The set of functions
  $\{\mu \mapsto \av{\mu, \varphi} \bigm| \varphi \in \C_{b,+}(\R^k)\}$, where
  $\C_{b,+}(\R^k):= \{\varphi \in \C_b(\R^k)|~\varphi \geq 0$\}, generates the Borel-$\sigma$-algebra of
  measures in $\mathcal M_1$ in the topology of weak convergence. Hence, the uniqueness of the law of the
  one-time marginals $\mathcal L(\mu_t)$ follows.  We can then establish uniqueness of the finite dimensional
  distributions using the martingale property.  An illustration of the argument for real-valued processes may
  be found in~\cite[Chapter 3.4]{Ethier1986a} and can be easily adapted to the present context.
\end{proof}

\begin{proposition}[Rigidity]
  \label{lemma:nonexistence}
  Any solution $(\mu_t)_{t \geq 0}$ to Equation~\eqref{eqn:F0} satisfies $\alpha \mu_t(A) \in \N$ for all sets
  $A \in \mathcal{B}(\R^k)$ and times $t >0$. Furthermore, there exist $n \in \N$ and
  $(x^i)_{i =1}^n \subset \R^k $ such that $\mu_0 = \frac{1}{\alpha}\sum_{i =1}^n\delta_{x^i}$.
\end{proposition}

\begin{proof}
  As shown in~\cite[Lemma 2.4]{Konarovskyi2019a}, if the moment generating function $g_X(s)$ of a nonnegative
  random variable $X$ is of the form $g(s) = \sum_{n = 0}^\infty p_k s^k$, then $X \in \N_0$ almost surely.
  We want to apply this statement to the non-negative random variable $\alpha\mu_t(A)$ for sets
  $A \in \mathcal{B}(\R^k)$.  Let us consider the moment generating function
  \begin{align*}
    g(s) = \E\left( e^{-|\ln(s)|\alpha \mu_t(A)} \right) = \exp\left( \av{\mu_0, \alpha \ln\left( P_t e^{-|\ln(s)|1_A} \right)} \right). 
  \end{align*}
  Again we follow argument of~\cite{Konarovskyi2019a}. Note that we need $P_t 1_A$ to be bounded away from 1
  for any $t >0$, in order to obtain that $g$ can be extended to a small interval around $0$ such that $g$ is
  analytic in 0.  Let $A \in \mathcal{B}(\R^k)$ be a set such that there exists a set $A_n$ of the exhaustion
  (Condition~\ref{assL:exhaust} of Assumption~\ref{assumptionsonL}) such that $A \subset A_n$. Hence,
  $P_t 1_A \leq P_t 1_{A_n} \leq c_n < 1$ is bounded away from $1$.  Analogously to the proof
  in~\cite{Konarovskyi2019a}, we now can extend
  $g(s) = \exp\left( \av{\mu_0, \alpha \ln\left( 1+ (s-1) P_t1_A \right)} \right)$ to
  $s \in \left(1-\frac{1}{c_n}, 1\right)$, where by assumption $1-\frac{1}{c_n} < 0$.  This bound implies that
  the extended function is analytic in $0$, so that we obtain the required expansion
  $g(s) = \sum_{k=0}^n p_k s^k + \mathcal{O}(s^{n+1})$ as $s \downarrow 0$.  Finally, this yields
  $\alpha \mu_t(A) \in \N$ for all $A \in \mathcal B(\R^k)$ satisfying $A \subset A_n$ for some $n \in
  \N$. Since $(A_n)_{n \in \N}$ is an exhaustion, the statement extends to arbitrary $A \in \mathcal B(\R^k)$.
  By the Portmanteau theorem, it then follows that $\alpha \mu_0(A) \in \N$ for all $A \in \mathcal B(\R^k )$
  bounded with $\mu_0(\partial A) = 0$, which implies that there exist $n \in \N$ and
  $(x^i)_{i =1}^n \subset \R^k $ such that $\mu_0 = \frac{1}{\alpha}\sum_{i =1}^n\delta_{x^i}$.
\end{proof}

\subsection{Interacting case $\F\ne 0$}
\label{sec:pf_for_f_ne_zero}

We treat the case $\F\ne 0$ as a perturbation of the $\F=0$ by means of the Girsanov theorem.  The proof
follows the ideas in~\cite[Theorem 2]{Konarovskyi2020a}, but some extra care is needed to the unboundedness of
the drift vector field $b$ in $L$. The critical remedy is an a priori bound on the second moments for
solutions of~\eqref{eqn:DKE1}.

\begin{lemma}
  \label{lemma:secondmoment}
  Let $b$ be of at most linear growth and $\F\colon\mathcal M_1 \times \R^k \to \R^k$ be continuous with
  $\sup_{\mu \in \mathcal M_1, z \in \R^k} |H(\mu,z)| < \infty$.  Assume that $(\mu_t)_{t \geq 0}$ is a
  solution to Equation~\eqref{eqn:DKE1} with initial condition $\mu_0 \in \cM^2_1$. Then $\mu_t \in \cM^2_1$
  almost surely for all $t \geq 0$.
\end{lemma}

\begin{proof} The choice of the test function $\varphi =1$ in the weak martingale formulation of
  \eqref{eqn:DKE1} reveals that the total mass $\mu_t(\R^k)$ is conserved, with value $1$. In particular, by
  the assumptions on $F$, there is some constant $c$ such that $\sup_{t\geq 0, z \in \R^k} |F(\mu_t,z)| <c$
  almost surely. Here and in below we use the symbol $c$ for a generic constant whose value may change from
  line to line. Next, for $R > 0$, choose $\psi_R \in \C^\infty_b(\R)$ a smooth approximation of the function
  $ r \mapsto (r+1) \wedge R$ on $\R_{\geq 0}$ such that $\psi_R(r) = R$ for $r \geq R$, $r \leq \psi_R(r)$
  for all $r \leq R$, and $ \psi_R' \leq 1$ and $ |\psi_R''| \leq c $ for uniformly in $R$.  Then
  $\varphi_R = \psi_R \circ e_2 \in \C^2_b(\R^k)$ and the chain rule for $L$ (cf.\
  Condition~\ref{assL:chainr}) together with~\ref{interact.L} of Assumption~\ref{assumption:interaction}
  yields
  \begin{equation*}|L \varphi_R(z)| \leq c \left(1 +\varphi_R(z)\right),\end{equation*}
  with a constant $c$ independent of $R$. 
  Now, since $\varphi_R \in \C^2_b(\R^k)$, by definition of $\mu_t$ the process $M_t(\varphi_R)$ is a
  continuous martingale and hence $\tau_{N,R} := \inf\{t \geq 0 \bigm|\av{\mu_t, \varphi_R}\geq N\}$ is a
  well-defined stopping time, with $\lim_{N \to \infty}\tau_{N,R} = \infty$ almost surely, for all $R >
  0$. Therefore,
 \begin{align*}
                \E\left( \av{\mu_{t\wedge\tau_N}, \varphi_R} \right) &= \E\left(\av{\mu_0, \varphi_R} + \int_0^{t\wedge\tau_N} \av{\mu_s, \alpha L\varphi_R+F_{\mu_s} \cdot \nabla \varphi_R}\dd s\right)\\
                &\leq \E\left( \av{\mu_0, \varphi_R} \right) + \int_0^t \E\left( \av{\mu_{s\wedge \tau_{N,R}}, c\left(1+ \varphi_R\right)} \right)\dd s,
\end{align*}
Since up to the stopping time $\tau_{N,R}$, the process is bounded,
$\E\left( \av{\mu_{s\wedge \tau_{N,R}},\varphi_R} \right)$ is a continuous function. Hence, by Gr\"onwall's
Lemma
\begin{align*}
  \E\left( \av{\mu_{t \wedge \tau_{N,R}}, \varphi_R} \right) &\leq \left( \av{\mu_0, \varphi_R} + \av{\mu_0}ct\right)e^{ct},
\end{align*}
where we use again that the mass is preserved almost surely.  Using Fatou's lemma again, we finally
obtain
\begin{align*}
  \E\left( \av{\mu_{t},\varphi_R} \right) &= \E\left(\liminf_{R} \av{\mu_{t \wedge \tau_{N,R}},\varphi_R} \right)
                                            \leq \left( \av{\mu_0, \varphi_R} + ct\right)e^{ct}.
\end{align*}
{Due to $e_2(z) = |z|^2 \leq \lim_{R \to \infty}\varphi_R$,  Fatou's Lemma yields}
\begin{align}
  \label{est:secmombound}
  \E\left( \av{\mu_t, e_2} \right) &\leq \left( \av{\mu_0, 1+ e_2} + ct\right)e^{ct}
\end{align} 
and so in particular $\av{\mu_t, e_2} < \infty$ almost surely for all $t \geq 0$.
\end{proof}

Next we need an It\^o formula along solutions to~\eqref{eqn:DKE1} for a sufficiently big class of functions on
measures. Let
\begin{multline*}
        \C^2(\cM) := \bigg\{E \in \C(\cM) \bigm| \frac{\delta E}{\delta \mu} \text{ and }\frac{\delta^2 E}{\delta \mu^2}  ~\text{exist and }~ \frac{\delta E}{\delta \mu} \in \C(\cM \times \R^k),\\ ~\frac{\delta^2 E}{\delta \mu^2} \in \C(\cM \times \R^k \times \R^k), \frac{\delta E}{\delta \mu}(\mu, \cdot) \in \C^2(\R^k) \text{ and}~ \frac{\delta^2 E}{\delta \mu^2} \in \C^2(\R^k \times \R^k) ~ \forall \mu \in \cM_F \bigg\}, 
        \end{multline*}
    and   let 
    \begin{multline*}
    \C^2_b(\cM) := \left\{ E \in \C^2(\cM)  \bigm|  \forall~ c > 0: \sup_{ \mu \in \mathcal M_1}|E(\mu)|< \infty,~ \right.\\ \sup_{ (\mu, z) \in \mathcal M_1 \times \R^k}\left|\frac{\delta E}{\delta \mu}(\mu, z)\right| < \infty,
    \left.\sup_{ (\mu, z_1, z_2) \in \mathcal M_1 \times \R^k \times \R^k}|E(\mu, z_1, z_2)| < \infty
    \right\}.
\end{multline*}

\begin{proposition} 
  \label{prop:itof_basic}
  For a solution $(\mu_t)_{t \geq 0}$ of equation~\eqref{eqn:DKE1} with initial condition $\mu_0 \in \cM^2_1$
  and a function $E \in \C_b^2(\mathcal M)$, the following process is a local martingale
  \begin{align*}
    M_t^E &=  E(\mu_t) - E(\mu_0) - \int_0^t \av{\mu_s, \alpha L \frac{\delta E}{\delta \mu}(\mu_s, \cdot) + F_{\mu_s} \cdot  \nabla \frac{\delta E}{\delta \mu}(\mu_s, \cdot) } \dd s\\
          &\qquad{} - \frac{1}{2} \int_0^t \av{\mu_s, \sigma \sigma^T:D  \frac{\delta^2 E}{\delta \mu^2}(\mu_s, \cdot,\cdot)} \dd s,
  \end{align*}
  where above $(D \psi)(z) := \partial_i \partial_{k+j} \psi(z, z)$ for $\psi \in \C^1(\R^k \times \R^k)$,
  with quadratic variation
  \begin{equation*}
    [M_\cdot^E]_t = \int_0^t \av{\mu_s, \Gamma\left( \frac{\delta E}{\delta \mu} \right)} \dd s. 
  \end{equation*}
 \end{proposition}

 \begin{proof} Using the standard It\^o-formula in $\R^n$, the claimed statement holds for functions on $\cM_1$
 of the type
 \begin{align*}
   P(\mathcal{M}_1) = \{\mu \mapsto f\left( \av{\mu, \varphi_1}, \dots, \av{\mu, \varphi_n} \right) \bigm| f \in \C^2(\R^n), ~\varphi_i \in \C_c^\infty(\R^k), i \in \{1, \dots, n\}, n \in \N\}.
 \end{align*}
 Second, since $\langle\mu_s\rangle =\langle\mu_s,1\rangle= \av{\mu_0}$ for almost all trajectories, it holds
 that
 \begin{equation*}
   [M^E_\cdot]_t = \int_0^t \av{\mu_s, \Gamma\left( \frac{\delta E}{\delta \mu} \right)} \dd s \leq c_E
   \av{\mu_0} t <\infty
 \end{equation*}
 and we conclude that $M^E$ is in fact a (square integrable) martingale.
         
 Next, according to~\cite[Theorem 5]{Konarovskyi2020a} for $E \in \C_b^2(\mathcal M)$ there exists a
 sequence of functions $E_n \in P(\mathcal M_1)$ such that $E_n \to E$ pointwise,
 $\frac{\delta E_n}{\delta \mu}(\mu, \cdot) \to \frac{\delta E}{\delta \mu}(\mu, \cdot) $ in $\C^2(\R^k)$ and
 $\frac{\delta^2 E_n}{\delta \mu^2}(\mu, \cdot, \cdot) \to \frac{\delta E^2}{\delta \mu^2}(\mu, \cdot, \cdot)$
 in $\C^2(\R^k \times \R^k)$. Here, $\C^2(\R^k)$ is equipped with the topology of uniform convergence of
 compact subsets.  Furthermore, the sequence $(E_n)_{n \in \N}$ is such that $(E_n)_{n \in \N}$ and its
 derivative are uniformly bounded on $\cN_c$, in the sense that there exists $c_E >0$ such that
\begin{align}
  \sup_{n \in \N}\sup_{\mu \in \cN_C}| E_n(\mu) | + \left\|\frac{\delta E_n}{\delta \mu}(\mu, z) \right\|_{2,\infty} + \left\|\frac{\delta^2 E_n}{\delta \mu^2}(\mu, z) \right\|_{2,\infty}\leq c_E,
  \label{eqn:bound}
\end{align}
where the norm is defined as
$||\varphi||_{2,\infty} = ||\varphi||_\infty + ||\nabla \varphi||_\infty + ||\nabla^2 \varphi||_\infty $ for
$\varphi \in \C^2_b(\R^k)$.
Furthermore, Condition~\ref{interact.L} of Assumption~\ref{assumption:interaction} yields the bound 
\begin{align*}
        \left|L \frac{\delta E_n}{\delta \mu}\right| \leq c (1+ |z|) \left\| \frac{\delta E}{\delta \mu}\right\|_{2,\infty}
\end{align*}
so that, using the stopping times
\[\rho_N := \inf\{t \geq 0| \av{\mu_t, e_2} \geq N \}\]
and uniform integrability, the convergence of $\frac{\delta^{l} E_n}{\delta \mu^l}(\mu)$ to
$\frac{\delta^{l} E}{\delta \mu^l}(\mu)$ for $l\in {0,1,2}$ together with their $\nabla$-derivatives up to
second order locally uniformly in $C^2(\R^k)$ and continuity of the coefficients of $L$ allows to pass to the
limit for $n \to \infty$ for the stopped martingales
\begin{align*}
  M_{t\wedge\rho_N}^{E_n} &=  {E_n}(\mu_{t\wedge\rho_N}) - {E_n}(\mu_0) - \int_0^{t\wedge\rho_N} \av{\mu_s, \alpha L \frac{\delta {E_n}(\mu_s)}{\delta \mu_s} + F_{\mu_s} \cdot  \nabla \frac{\delta {E_n}}{\delta \mu} } \dd s\\
                          &- \frac{1}{2} \int_0^{t\wedge\rho_N} \av{\mu_s, \sigma \sigma^T:D  \frac{\delta^2 {E_n}}{\delta \mu^2}} \dd s
\end{align*}
with associated quadratic variation processes
\begin{align*}
  [M^{E_n}_\cdot]_{t\wedge \rho_N} &=  \int_0^{t\wedge \rho_N} \av{\mu_s, \Gamma\left( \frac{\delta E_n}{\delta \mu} \right)}\dd s.
\end{align*}
Since $\rho_N\nearrow \infty$ by Lemma~\ref{lemma:secondmoment}, the claim is proved.
\end{proof}

The previous It\^o formula extends to a wider class of functions. To this aim, for $c>0$ let
$\cM^2_{1,c}= \{\mu \in \cM^2_1 \bigm| \langle \mu,z^2\rangle <c\}$.

\begin{definition}
  \label{def:c2b_etc}
  A function $E\colon \cM^2 \mapsto \R$ belongs to the class $\C^2_{loc}(\cM^2)$ if
  \begin{itemize}
  \item $\frac{\delta E}{\delta \mu}(\mu,.) \in \C^2(\R^k) $ and
    $\frac{\delta^ 2 E}{\delta \mu^2}(\mu,.,.) \in \C^2(\R^k\times\R^k) $ exist for all $\mu \in \cM^2_1$;
  \item $E(.)$, $\frac{\delta E}{\delta \mu}$, and $\frac{\delta^ 2 E}{\delta \mu^2}$ as well as their first
    and second order spatial derivatives $\nabla^{l}\frac{\delta E}{\delta \mu}(.,)$ for $l\in \{1,2\}$ are
    (jointly) continuous on $\cM^2_{1,c}$, $\cM^2_{1,c}\times \R^k$ respectively $\cM^2_{1,c}\times \R^k\times \R^k$.
  \end{itemize}
  We say that a function $H\colon \cM^2_1\times \R^n \to \R $ is of \emph{at most linear growth, locally
    uniformly} on $\cM^2_1$, if for all $c>0$ there is a constant $K>0$ such that
  \[ H(\mu, z) \leq K (1+ |z|) \mbox{ for all } (\mu,z) \in \cM^2_{1,c} \times \R^n.\]
  $H$  is called \emph{bounded uniformly} on $\cM_1$ if  $\sup_{\mu \in \mathcal M_1, z \in \R^k } |H(\mu,z)| < \infty$. 
\end{definition}

\begin{corollary}
  \label{prop:itof_extended}
  The assertion of Proposition~\ref{prop:itof_basic} remains true if the function $E$ is in the class
  $\C^2_{loc}(\cM^2)$, provided $\frac{\delta E}{\delta \mu}$ and $\frac{\delta^ 2 E}{\delta \mu^2}$ together
  with their first and second order spatial derivatives are of at most linear growth, locally uniformly on
  $\cM^2_1$.
\end{corollary}

\begin{proof} Using a sequence of increasing smooth cut-off functions $\chi_m\in C^\infty_c(\R^k)$ such that
  $0\leq \chi_m\leq 1 $ with $\chi_m \nearrow 1$, define $\chi_m \mu(A) := \av{\mu, \chi_m 1_A}$ for
  $A \in \mathcal{B}(\R^k)$.  We can approximate $E$ by $E_m\in \C^2_b(\cM)$, $E_m(\mu):=E(\chi_m\mu)$, to
  which the previous Proposition~\ref{prop:itof_basic} applies. One checks that indeed $E_m\in \C^2_b(\cM)$,
  using
\begin{align*}
	\frac{\delta E_m}{\delta \mu}(\mu, z) &= \chi_m(z) \frac{\delta E }{\delta \mu}(\chi_m \mu, z) \quad \text{and}\\
	\frac{\delta^ 2 E_m}{\delta \mu^2}(\mu, z_1, z_2) &= \chi_m(z_1)\chi_m(z_2) \frac{\delta^ 2 E }{\delta \mu^ 2}(\chi_m \mu, z_1, z_2). 
\end{align*}
Taking advantage of the same sequence of stopping times $\rho_N$ as in the proof of
Proposition~\ref{prop:itof_basic} and using the regularity and and linear boundedness of $E$ and its
derivatives, one can pass to the limit $m\to \infty$ in the stopped It\^o formula for $E_m$, following the
same arguments as in the proof of Proposition~\ref{prop:itof_basic}. The convergence to a local martingale
follows from the convergence of the second moments
\begin{align*}
\E\left( \sup_{t \in [0,T]}  \bigm| M_{t \wedge \rho_N}^{E_n} - M_{t \wedge \rho_N}^{E_m}|^2\right) %
\xrightarrow{n,m \to \infty} 0
\end{align*}
using the Burkholder-Davis-Gundy inequality, the definition of $\rho_N$ and the linear bound on $E$.
\end{proof}

\begin{proposition}
  \label{prop:girsanovfordke}
  Let the interaction $\F$ satisfy Assumption~\ref{assumption:interaction} and let $(\mu_t)_{t \geq 0}$ be a
  solution to Equation~\eqref{eqn:DKE1} with initial condition $\mu_0 \in \mathcal M^2_1$.  Define $\mathbb Q$
  by $\dd \mathbb Q = \dd \mathbb Q = e^{-M^G_t - \frac{1}{2}[M^G]_t} \dd \P$ on $\mathcal F_t$, for
  $t \geq 0$. Then, $\mu_t$ solves
  \begin{align*}
    \dd \mu_t = \alpha L^* \mu_t \dd t + \nabla \cdot \left(\sqrt{\mu_t} \sigma \dd W_{z,t}\right)
  \end{align*}
  on $(\Omega, \mathcal F, (\mathcal F_t)_{t \geq 0}, \mathbb Q)$, where $L^*$ denotes the dual of $L$.       \label{prop:Girsanov}
\end{proposition}

\begin{proof}

  The statement is an immediate consequence of Girsanov's theorem using the (Dol\'eans-Dade) stochastic
  exponential $\mathcal E(M^{\alpha G})=e^{-M^{\alpha G}_\cdot - \frac{1}{2}[M^{\alpha G}]_\cdot}$ as density,
  where $M^{\alpha G}$ is the local martingale obtained by applying Corollary~\ref{prop:itof_extended} to the
  functional $G$ of Condition~\ref{F-F1} of Assumption~\ref{assumption:interaction}. The Novikov condition is
  trivially satisfied due to $\mu_t(\R^d) = \mu_0(\R^d)$ a.s.\ and the locally uniform boundedness of
  $\sigma^T \nabla \frac{\delta G}{\delta \mu}$ follows from Condition~\ref{F-F2}.
\end{proof}

\begin{proof}[Proof of the main Theorem~\ref{thm:mainthm} for $\F\ne 0$] By
  Proposition~\ref{prop:girsanovfordke}, if $\mathbb Q$ denotes the unique law of the solution to
  Equation~\eqref{eqn:DKE1} for $\F=0$ and $\mu_\cdot$ is a solution to~\eqref{eqn:DKE1} for $\F\ne 0$, for
  every measurable path functional $\Phi$ on the set of continuous $\cM_1$-valued paths it holds that
  \begin{equation*}
    \mathbb E\bigl(\mathcal E(M^G(\mu_\cdot)) \Phi(\mu_\cdot)\bigr) =
    \E_{\mathbb Q}\bigl(\Phi(\mu_\cdot)\bigr).
  \end{equation*} 
  Using the statement with $\Phi(\mu_\cdot)$ replaced by
  $\tilde \Phi(\mu_\cdot) = \mathcal E(M^G(\mu_\cdot))^{-1}\Phi(\mu_\cdot)$ yields the claim.
\end{proof}

\section*{Acknowledgments}
JZ gratefully acknowledges funding by the US Army Research Office, grant W911NF2310230.  FM thanks Lorenzo
Dello Schiavo for interesting discussions about the Dean-Kawasaki equation. All authors are grateful for
extensive comments by two reviewers, who helped improving the manuscript significantly.

\printbibliography
\end{document}